\documentclass[twoside]{article}

\usepackage[utf8]{inputenc}
\usepackage[english]{babel}
\usepackage{amsthm}
\usepackage{amsmath}
\usepackage{graphicx}
\usepackage{amsfonts}

\theoremstyle{plain}
\newtheorem{theorem}{Theorem}

\newtheorem{proposition}{Proposition}

\theoremstyle{definition}
\newtheorem{definition}{Definition}

\begin{document}
\title{On the matching arrangement of a graph,\\
improper weight function problem and its application.}
\author{A.~I.~Bolotnikov,~A.~A.~Irmatov}
\date{}
\maketitle

\begin{abstract} 
This article presents examples of an application of the finite field method for the computation of the characteristic polynomial of the matching arrangement of a graph.  Weight functions on edges of a graph with weights from a finite field are divided into proper and improper functions in connection with proper colorings of vertices of the matching polytope of a graph.  An improper weight function problem is introduced,  a proof of its NP-completeness is presented, and a knapsack-like public key cryptosystem is constructed based on the improper weight function problem.
\end{abstract}

\section{Introduction}
In \cite{bol_first}, the matching arrangement of hyperplanes $MA(G)$
for a graph $G(V,E), |E| = n$ without loops and multiple edges was
introduced as the arrangement in $\mathbb{R}^n$ that contains a
hyperplane with the equation  $x_{i_1} -
x_{i_2}+...+(-1)^{k+1}x_{i_k} = 0$ for every sequence of edges
$(e_{i_1},..., e_{i_k})$ that form either a simple path of arbitrary
length or a simple cycle of even length in $G$. The matching
arrangement has the following property: if you arbitrarily select
two vectors from the same region of the arrangement, use each vector
in turn as a weight vector on the edges of graph $G$, and find the
maximum-weight matching for each vector, then the maximum matching
will be the same  for both vectors \cite{bol_first}. In addition,
there is a one-to-one correspondence between regions of the matching
arrangement and LP-orientations of the matching polytope, a special
polytope used in solving the maximum matching problem using linear
programming methods \cite{bol_second}.

In \cite{zasl}, Thomas Zaslavsky showed that  the number of regions
in a hyperplane arrangement depends on the partially ordered set of
intersections of hyperplanes in the arrangement. Zaslavsky expressed
this number through the characteristic polynomial of the hyperplane
arrangement. Thus, if we calculate the characteristic polynomial of
a hyperplane arrangement, we can find its number of regions.

The results on the properties  of the characteristic polynomial of a
matching arrangement, as well as on calculating the characteristic
polynomial for the case when graph $G$ is a tree, are also presented
in \cite{bol_first}. In this paper, we will consider examples of
calculating the characteristic polynomial of a matching arrangement
using the finite field method. This method was developed by
Athansiadis \cite{atn} based on the results of Crapo and Rota
\cite{crapo-rota}, and allows us to calculate the characteristic
polynomial of a hyperplane arrangement if all equations of the
hyperplanes in the arrangement have rational coefficients. The
method is based on the following result \cite[Theorem 2.1]{atn}:

{\itshape Let $A$ be any subspace arrangement in  $\mathbb{R}^n$
defined over the integers and q be a large enough prime number.
Then, $\chi_{A}(q) = |F_{q}^{n} \setminus \cup A|$. Equivalently,
identifying $F_{q}^{n}$ with $\{0, 1, ..., q-1\}^{n}$, $\chi_A(q)$
is the number of points with integer coordinates in the cube $\{0,
1, ..., q-1\}^{n}$ which do not satisfy mod $q$ the defining
equations of any of the subspaces in $A$. }

The second section of this paper presents  results on calculating
the characteristic polynomials of the matching arrangement for some
families of graphs using the finite field method. The definitions of
a hyperplane arrangement, an intersection lattice and a
characteristic polynomial that were used in the second section can
be found in \cite{stanley}, which also contains the proof of the
deletion-restriction lemma and the description of the finite field
method that uses powers of prime.

In the third section, we introduce  the definition of a proper and
improper weight function, state the improper weight function
problem, prove its NP-completeness, and describe the relationship
between proper weight functions and the matching arrangement.  The
NP-completeness is proved with the use of 3-satisfiability, or 3-SAT
problem. All the necessary definitions regarding the 3-SAT problem,
as well as proof of its NP-completeness, can be found in
\cite{garey-johnson}.

The fourth section of the paper provides a theoretical description
of the Bolotnikov--Irmatov cryptosystem that uses the
NP-completeness of the improper weight function problem.

The results from the second and the third sections are obtained by
A.I.~Bolotnikov. The cryptosystem in the fourth section is developed
by A.I.~Bolotnikov and A.A.~Irmatov.

\section{The characteristic polynomial of the matching arrangement for some classes of graphs}

\begin{definition}
Let $A$ be an arrangement of hyperplanes in $\mathbb{R}^n$.  A
partially ordered set (or poset) of intersections $L(A)$ is defined
the following way:  elements of $L(A)$ are the space $\mathbb{R}^n$
itself, and all nonempty intersections of hyperplanes in $A$, and $x
\leq y$ in $L(A)$, if $x \supseteq y $.
\end{definition}
Each element $x$ of the set $L(A)$ is  an affine subspace of
$\mathbb{R}^n$. Let $dim(x)$ be the dimension of $x$. If all
hyperplanes of $A$ intersect in one point, then $L(A)$ is a lattice
and is referred to as a lattice of intersections instead.

\begin{definition}
\textit{M\"{o}bius function} $\mu: P \times P \rightarrow
\mathbb{Z}$ for a partially ordered set $P$ is defined by following
conditions:
\begin{enumerate}
    \item $\mu (x,x) = 1$ $ \forall x \in P$.
    \item $\mu (x,y) = -\sum_{x \leq z < y} \mu (x,z)$ $ \forall x<y \in P $.
    \item $\mu (x,y) = 0$, if $x$ and $y$ are not comparable.
\end{enumerate}
\end{definition}
If $P$ has a minimal element $\hat 0$,  then we write $\mu(x) := \mu
(\hat{0}, x)$.
\begin{definition}
\textit{A characteristic polynomial} of a hyperplane arrangement A
is defined by the following formula
\[ \chi_{A} (t) = \sum_{x \in L(A)} \mu(x)t^{dim(x)}.\]
\end{definition}

\begin{proposition} Let $G$ be a graph that consists of one cycle  of length $n$, and let $\chi_{MA(G)}(x)$ be
the characteristic polynomial of the matching arrangement $MA(G)$.
Then,
\begin{itemize}
\item if $n$ is even, then $\chi_{MA(G)}(x) = (x-1)(x-n)^{n-1}$.
\item if $n$ is odd, then  $\chi_{MA(G)}(x) = (x-1)(x-3)(x-5)...(x-(2n-3))(x-(n-1))$.
\end{itemize}
\end{proposition}

\begin{proof}
Let graph $G$ consist of one cycle of length $n$. We arbitrarily
choose a vertex $A$ in graph $G$ and number the edges clockwise,
counting the edges from $A$.

\begin{figure}[h]
     \centering
     \includegraphics[width=0.3\textwidth]{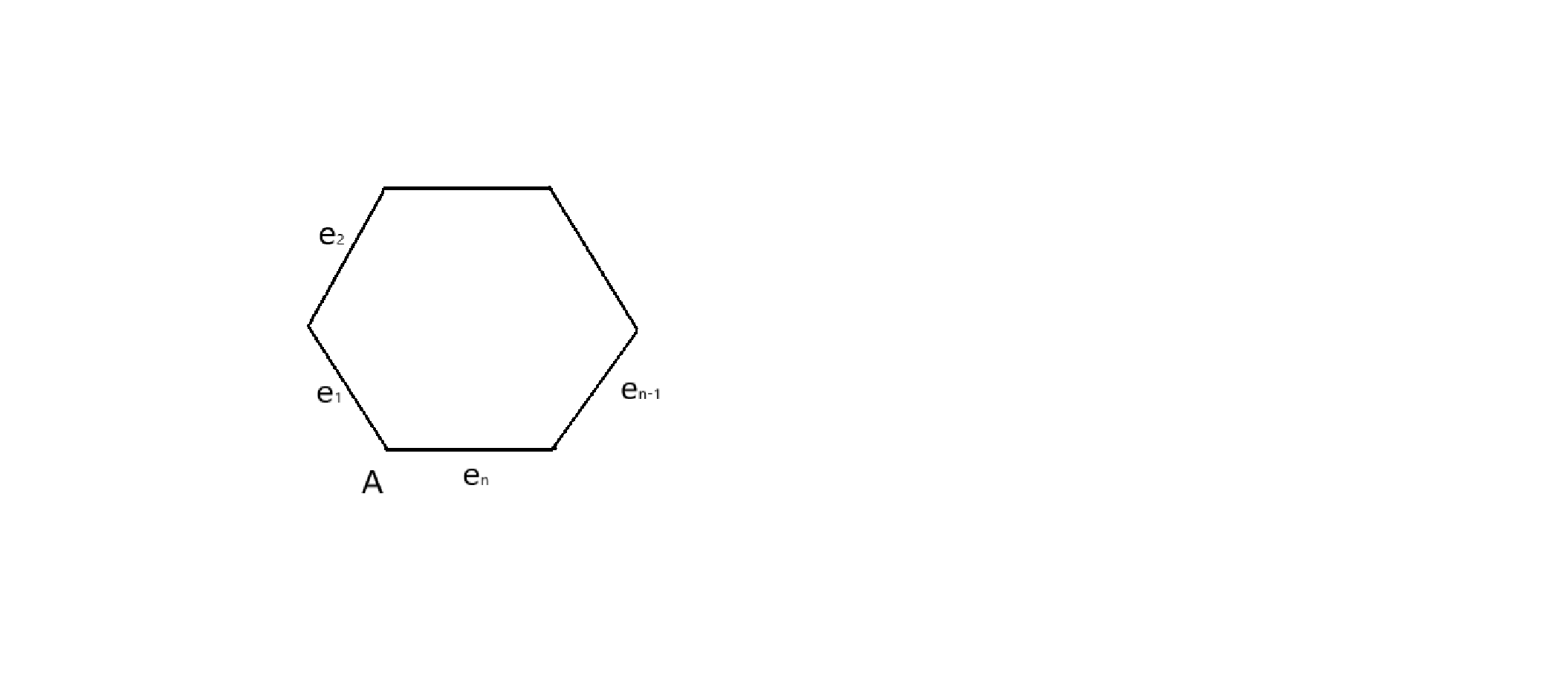}
     \caption{A cycle of length $n$}
     \label{fig:my_label}
 \end{figure}

Let us first consider the case when $n$ is even. Consider  the
linear operator $W: \mathbb{R}^{n} \rightarrow \mathbb{R}^{n}$ with
the following matrix. The main diagonal of the matrix, as well as
the next diagonal to the right, alternate between 1 and $-1$
(starting with 1), and there are 0s in all other places.

\[ \left( \begin{matrix}
1 & 1 & 0 & 0 & 0 & 0  \\
0 & -1 & -1 & 0 & 0 & 0 \\
0 &  0 &  1 & 1 & 0 & 0 \\
0 &  0 &  0 & -1 &-1 & 0\\
0 &  0 &  0 &  0 &  1 & 1 \\
0 &  0 &  0 &  0 &  0 &-1 \\ 
\end{matrix} \right) \]

The vector $ t = (1, -1,...,(-1)^{k+1},0,...,0)$  is the normal
vector to the hyperplane from the arrangement $MA(G)$, which
corresponds to a path of length $k$ leaving $A$ and going clockwise.
The operator $W$ maps this vector to vector  $d_k =
(0,...,0,1,0,...,0)$, in which 1 is at the $k$th position. For an
arbitrary vertex $B$ of graph $G$, the vector $t_B =  (1,
-1,...,(-1)^{r+1},0,...,0)$, where $r$ is the length of the path
from $A$ to $B$ clockwise, is normal to the hyperplane corresponding
to the path from $A$ to $B$ in a clockwise direction. $t_A = (1,
-1,...,1,-1) $ will denote the normal to the hyperplane
corresponding to the entire cycle. Consider an arbitrary path in
graph $G$ that leads clockwise from vertex $B$ to vertex $C$. If $B
= A$, then the vector $t_C$ will be normal to the hyperplane
corresponding to this path. If $C = A$ or the path does not pass
through $A$, then $t_C - t_B$ will be normal to the corresponding
hyperplane. If the path passes through $A$ but does not start or end
there, then $t_A - t_B + t_C$ will be normal to the corresponding
hyperplane. Operator $W$ maps $t_B$ to $d_r$, $t_C - t_B$ to
$d_{k_2} - d_{k_1}$, where $k_2 > k_1$, and $t_A - t_B + t_C$ to
$d_n - d_{k_1} + d_{k_2}$, where $k_1 > k_2$.

Consider the hyperplane arrangement $S$ that consists of the following hyperplanes:
\begin{enumerate}
\item $\forall i, 1\leq i\leq n: x_i = 0 $.
\item $\forall i_1, i_2, 1\leq i_1 \leq i_2 \leq n: x_{i_2} - x_{i_1} = 0$.
\item $\forall i_1, i_2, 1\leq i_1 \leq i_2  < n: x_n - x_{i_2} + x_{i_1} = 0$.
\end{enumerate}

Note that since there is the invertible linear operator $W$ that
maps the set of normal vectors  to hyperplanes of  $MA(G)$ to the
set of normal vectors to hyperplanes of $S$, their intersection
lattices are isomorphic and their characteristic polynomials are
equal. We will compute the characteristic polynomial $\chi_{S}(x)$
using the finite field method. For a sufficiently large prime number
$p$, $\chi_{S}(p)$ is equal to the number of vectors $(x_1,...,x_n),
x_i \in F_p$, that do not satisfy any of the equations:
\begin{enumerate}
\item $\forall i, 1\leq i\leq n: x_i = 0 $.
\item $\forall i_1, i_2, 1\leq i_1 \leq i_2 \leq n: x_{i_2} - x_{i_1} = 0$.
\item $\forall i_1, i_2, 1\leq i_1 \leq i_2  < n: x_n - x_{i_2} + x_{i_1} = 0$.
\end{enumerate}

We will find the number $\chi_{S}(p)$ combinatorially.  First, we
choose a value for $x_n$. This can be done in $p-1$ ways. Let us say
we chose $x_n = k$, $k \neq 0$. Then, the remaining $x_i$ should be
selected so that they do not satisfy any of the following
conditions:
\begin{enumerate}
\item $\forall i, 1\leq i\leq n-1: x_i = 0 $.
\item $\forall i, 1\leq i\leq n-1: x_i = k $.
\item $\forall i_1, i_2, 1\leq i_1 \leq i_2 \leq n-1: x_{i_2} - x_{i_1} = 0$.
\item $\forall i_1, i_2, 1\leq i_1 \leq i_2 \leq n-1: x_{i_2} - x_{i_1} = k$.
\end{enumerate}

Let us choose a weak ordered  partition $\pi = (B_1,..., B_{p-n})$
of the set $\{1,...,n-1\}$ into $p-n$ blocks (``weak'' means that
some of the blocks can be empty), that is $\bigcup B_i =
\{1,...,n-1\}$ and $B_i \cap B_j = \emptyset , i\neq j$. Arrange all
the elements of $F_p$ in a row as follows: $0, k ,2k, ... ,(p-1)k$.
The first two elements of this row are ``locked'': none of the $x_i$
can take these two values. Starting with the third element in this
row, elements of $B_1$ are placed one by one in the descending order
onto the elements of the row without any gaps. Once all elements of
$B_1$ are placed, a gap with the size of one element of the row is
made, and elements of $B_2$ are placed on the following elements of
the row in the same descending order. We keep placing elements of
$B_i$ onto the elements of the row in the descending order and
making one-element gaps when moving from $B_i$ to $B_{i+1}$, until
all elements from $\{1,..., n-1\}$ are placed. 

Now for every $i$, $
0\leq i \leq n-1$, we choose the element of the row, in which we
placed $i$, as $x_i$. Conditions $ x_i = 0 $ and $ x_i = k $ aren't
satisfied for any $i$, because the first two elements of the row
were left ``locked''. Condition $x_{i_2} - x_{i_1} = 0$ is not
satisfied for any $i_1$ and $i_2$, because we put no more, then one
element of $\{1,...,n-1\}$ onto each element of the row. Condition
$0\leq i_1 \leq i_2 \leq n-1: x_{i_2} - x_{i_1} = k$ is not
satisfied for any $i_1 \leq i_2$, because for any $B_i$ the elements
of $B_i$ were placed in the descending order, and there was a
one-element gap between elements of $B_i$ and $B_{i+1}$. Therefore,
a vector $(x_1,..., x_{n-1})$ is constructed for a given partition
$\pi$, and it does not satisfy any of the four conditions.

Conversely, given a vector $(x_1,..., x_{n-1})$ that does not
satisfy any of the four conditions, a placement of $\{1,...,n-1\}$
onto the row $0, k, 2k,...,(p-1)k$ can be recovered, and in turn a
partition $\pi$ can be recovered  from this placement. This means
that the number of vectors that do not satisfy any of the four
conditions is equal to the number of weak ordered partitions of
$\{1,..., n-1\}$ into $p-n$ blocks. This number is equal to the
number of functions $f:[n-1] \rightarrow [p-n]$. There are exactly
$(p-n)^{n-1}$ of such functions. And considering the choice of $x_n$
on the first step, the resulting formula for the characteristic
polynomial of $S$ is $\chi_{S}(p) = (p-1)(p-n)^{n-1}$. Thus, in case
when graph $G$ consists of one cycle of even length,
$\chi_{MA(G)}(x) = (x-1)(x-n)^{n-1}$.

Now let $n$ be odd. Let $L$ be the  following arrangement of
hyperplanes: $L = MA(G)\cup \{H\}$, where $H$ is a hyperplane with
the equation $x_1 - x_2 +...-x_{n-1} + x_n = 0$.

Let the linear operator  $W: \mathbb{R}^{n} \rightarrow
\mathbb{R}^{n}$ be defined similar to the operator in the case of an
even cycle. The main diagonal of the matrix, as well as the next
diagonal to the right, alternate between 1 and $-1$ (starting with
1), and there are 0s in all other places. We will use $d_k$ and
$t_B$ to denote the same vectors as in the case of an even graph,
with $t_A = (1,-1,...,1,-1,1)$ now being a normal vector to $H$. The
representations of normal vectors to hyperplanes of $L$ as linear
combinations of vectors $t_A$, $t_B$, and  $t_C$ are mostly the
same, the only difference is that if the path from $B$ to $C$
passes through $A$ but does not start or end there, then $t_A - t_B
- t_C$ will be normal to the corresponding hyperplane.

Consider the hyperplane arrangement $S$ that consists of the following hyperplanes:
\begin{enumerate}
\item $\forall i, 1\leq i\leq n: x_i = 0 $.
\item $\forall i, 1\leq i_1 \leq i_2 \leq n: x_{i_2} - x_{i_1} = 0$.
\item $\forall i, 1\leq i_1 \leq i_2  < n: x_n - x_{i_2} - x_{i_1} = 0$.
\end{enumerate}

Note that since there is the invertible linear operator  $W$ that
maps the set of normal vectors to hyperplanes of  $L$ to the set of
normal vectors to hyperplanes of $S$, their intersection lattices
are isomorphic and their characteristic polynomials are equal. We
will compute the characteristic polynomial $\chi_{S}(x)$ using the
finite field method. For a sufficiently large prime number $p$,
$\chi_{S}(p)$ is equal to the number of vectors $(x_1,...,x_n), x_i
\in F_p$, that do not satisfy any of the equations:
\begin{enumerate}
\item $\forall i, 1\leq i\leq n: x_i = 0 $.
\item $\forall i, 1\leq i_1 \leq i_2 \leq n: x_{i_2} - x_{i_1} = 0$.
\item $\forall i, 1\leq i_1 \leq i_2 < n: x_n - x_{i_2} - x_{i_1} = 0$.
\end{enumerate}

We will find the number $\chi_{S}(p)$  combinatorially. First, we
choose a value for $x_n$. This can be done in $p-1$ ways. Let us say
we chose $x_n = k$, $k \neq 0$. Then, the remaining $x_i$ should be
selected so that they do not satisfy any of the following
conditions:
\begin{enumerate}
\item $\forall i, 1\leq i\leq n-1: x_i = 0 $.
\item $\forall i, 1\leq i\leq n-1: x_i = k $.
\item $\forall i_1, i_2, 1\leq i_1 \leq i_2 \leq n-1: x_{i_2} - x_{i_1} = 0$.
\item $\forall i_1, i_2, 1\leq i_1 \leq i_2 \leq n-1: x_{i_2} + x_{i_1} = k$.
\end{enumerate}

Let us consider two possible cases.

Case 1: $\forall i, 0 \leq i \leq n-1,~x_i \neq k/2$. In this case,
we can consecutively choose values for all $x_i$. There are $p-3$
possible values for $x_1$. After a value is chosen for $x_1$, there
are $k-5$ possible values for $x_2$: if value $t$ is chosen for
$x_1$, then we cannot choose $t$ and $k-t$ for $x_2$. After a value
is chosen for $x_2$, there are $p-7$ possible values for $x_3$, and
so on. In total, in this case there are
$(p-3)(p-5)(p-7)...(p-(2n-1))$ possible ways to choose values for
$x_1,...,x_{n-1}$.

Case 2: for some $i$ $x_i = k/2$. In this case, we start with
choosing in $n-1$ ways, which of $x_i$ is equal to $k/2$, and then
for the rest of  $x_i$ the values are chosen the same way, as in
Case 1. Therefore, in this case there are
$(n-1)(p-3)(p-5)...(p-(2n-3))$ ways to choose values for
$x_1,...,x_{n-1}$.

After adding the numbers from two cases,  and considering the choice
of the value for $x_n$ on the first step, the following formula for
the characteristic polynomial of $S$ is obtained: $\chi_{S}(p) =
(p-1)(p-3)(p-5)...(p-(2n-3))(p-n)$.

Let $\hat S$ be the arrangement  $\hat S = S \setminus {\hat H}$,
where $\hat H$ is the arrangement with the equation $x_n = 0$. The
invertible linear operator $W$ maps the set of normal vectors to
hyperplanes of  $MA(G)$ to the set of normal vectors to hyperplanes
of $\hat S$ , which means that their intersection lattices are
isomorphic and their characteristic polynomials are equal.

The characteristic polynomial  of $\hat S$ can be found with the use
of the deletion-restriction lemma:  $\chi_{\hat S}(x) = \chi_{S}(x)
+ \chi_{S_{\hat H}}(x)$, where $S_{\hat H}$ is the arrangement of
hyperplanes in $\hat H$ that consists of intersections of $\hat H$
with other hyperplanes of $S$. The following are the equations of
hyperplanes of $S_{\hat H}$:
\begin{enumerate}
\item $\forall i, 1\leq i\leq n-1: x_i = 0 $
\item $\forall i_1, i_2, 1\leq i_1 \leq i_2 \leq n-1: x_{i_2} - x_{i_1} = 0$
\item $\forall i_1, i_2, 1\leq i_1 \leq i_2 \leq n-1: x_{i_2} + x_{i_1} = 0$.
\end{enumerate}

We will find $\chi_{S_{\hat H}}(x)$ using the finite field method.
For a sufficiently large prime  number $p$, $\chi_{S}(p)$ is equal
to the number of vectors $(x_1,...,x_n), x_i \in F_p$, that do not
satisfy any of the equations above. Let us choose values for $x_i$
consecutively. There are $p-1$ ways to choose a value for $x_1$,
there are $p-3$ ways to choose a value for $x_2$ after choosing a
value for $x_1$, and so on. Therefore, $\chi_{S_{\hat H}}(p) =
(p-1)(p-3)(p-5)...(p-(2n-3))$.

After applying the deletion-restriction lemma,  the following result
is obtained: in case when $G$ consists of one cycle of odd length,
$\chi_{MA(G)(x)} = (x-1)(x-3)(x-5)...(x-(2n-3))(x-(n-1))$.
\end{proof}

\begin{proposition}
Let graph $G$ consist of two subgraphs  that intersect in one vertex
-- one subgraph is a complete graph $K_3$ and the other is the 
``tail'' -- a path of an arbitrary length. Then,  $\chi_{MA(G)}(x) = 
(x-1)(x-3)(x-4)...(x-(n-2))(x-(n-1))(x-(n-1))(x-n)$. 
\end{proposition}

\begin{proof}
Let graph $G$ consist of two subgraphs  that intersect in one vertex
-- one subgraph is a complete graph $K_3$ and the other is the
``tail''-- a path of an arbitrary length. Let us number the edges of
$G$ in the following order: first we number the edges of $K_3$, then
we number the edges of the tail, starting with the edge that is
adjacent to $K_3$.

\begin{figure}[h]
     \centering
     \includegraphics[width=\textwidth]{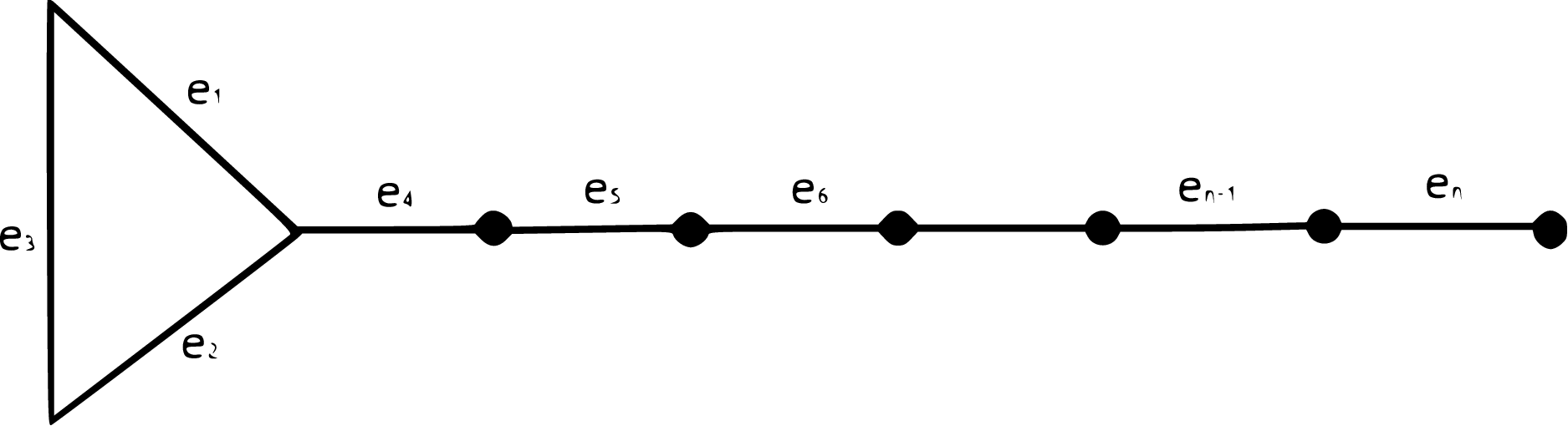}
     \caption{Graph $G$}
     \label{fig:my_label}
 \end{figure}

Let $W$ be a linear operator $W: \mathbb{R}^{n} \rightarrow
\mathbb{R}^{n}$ with the matrix shown on Fig.~\ref{head_matr}.

\begin{figure}[h]
     \centering
     \includegraphics[width=0.2\textwidth]{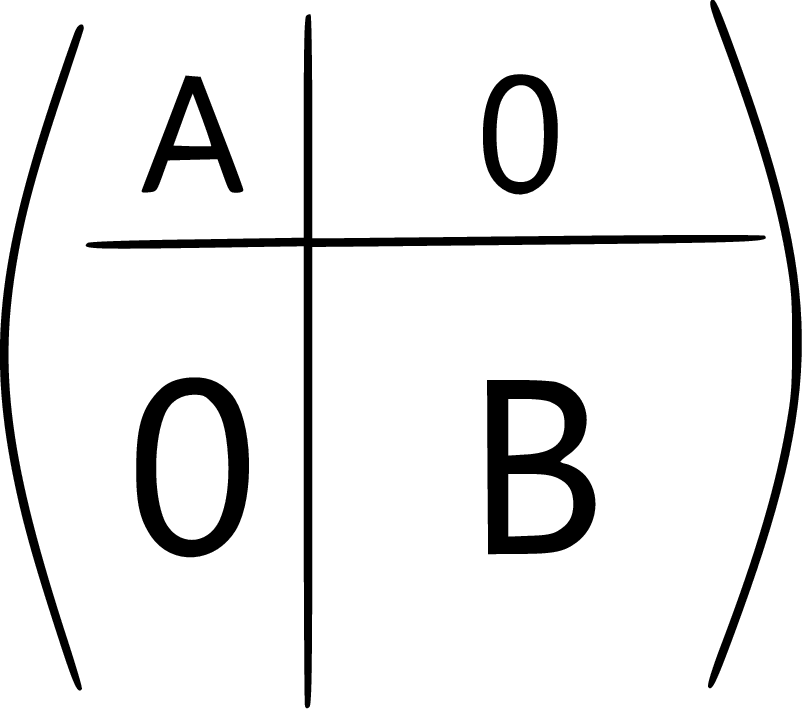}
     \caption{Matrix of the operator $W$}
     \label{head_matr}
 \end{figure}

Submatrix $A$ is this matrix of a size 3 by 3:
\[ \left( \begin{matrix}
1 & 0 & 0 \\ 0 & 1 & 1 \\ 0 & 0 & -1 \\
\end{matrix} \right) \]

The main diagonal of the submatrix $B$, as well  as the next
diagonal to the right, alternates between 1 and $-1$ (starting with
1), and there are 0s in all other places.

\[ \left( \begin{matrix}
1 & 1 & 0 & 0 & 0 & 0 & 0 \\
0 & -1 & -1 & 0 & 0 & 0 & 0\\
0 &  0 &  1 & 1 & 0 & 0& 0 \\
0 &  0 &  0 & -1 &-1 & 0 & 0\\
0 &  0 &  0 &  0 &  1 & 1 & 0 \\
0 &  0 &  0 &  0 &  0 &-1  & -1\\
0 &  0 &  0 &  0 &  0 &0   &1 \\
\end{matrix} \right) \]

Let $d_1 = (1,0,...,0)$, $d_2 = (0,1,0,...,0)$, $d_3 =
(0,1,-1,0,...,0)$,  $d_i = (0,0,0,1,-1,...,(-1)^{i},0,...,0)$, where
$(-1)^{i}$ stands on the $i$th place, $i >3$. Operator $W$ maps
$d_i$ into a vector $(0,...,0,1,0,...,0)$, where 1 stands on the
$i$th place.

Let us consider three possible cases for paths in $G$.

Case 1: The path is entirely inside the $K_3$ subgraph.  There are
only 6 paths that belong to this case: $\{e_1\}, \{e_2\}, \{e_3\},
\{e_1, e_2\}, \{e_2, e_3\}, \{e_3, e_1\}$.  $d_1, d_2, d_3 - d_2,d_2
- d_1,d_3, d_3 - d_2 + d_1$ are normal vectors to the corresponding
hyperplanes.

Case 2: The path is entirely inside the tail.  If the path contains
$e_4$, then $d_i$ will be a normal vector to the corresponding
hyperplane, otherwise $d_j - d_i, j>i$ will be a normal to the
corresponding hyperplane.

Case 3: The path goes from $K_3$ into  the tail. The following are
all possible cases of the intersection of this path and the $K_3$
subgraph. If this intersection is $\{e_1\}$, then $d_i - d_1$ is a
normal vector to the corresponding hyperplane. If $\{e_2\}$ is the
intersection, then $d_i - d_2$ will be the normal vector. If $\{e_2,
e_3\}$ is the intersection, then $d_i - d_3$ will be the normal
vector. If $\{e_3, e_1\}$ is the intersection, then $d_i - d_1 + d_2
- d_3$ is the normal vector.

Let $S$ be the arrangement that consists  of the following
hyperplanes:
\begin{enumerate}
\item $\forall i, 1\leq i\leq n: x_i = 0 $.
\item $\forall i_1, i_2, 1\leq i_1 \leq i_2 \leq n ,(i_1, i_2) \neq (1,3): x_{i_2} - x_{i_1} = 0$.
\item $\forall i, 4\leq i \leq n: x_i - x_1 + x_2 - x_3 = 0$.
\item $x_1 - x_2 + x_3 = 0$.
\end{enumerate}

Note that since there is the  invertible linear operator $W$ that
maps the set of normal vectors to hyperplanes of  $MA(G)$ to the set
of normal vectors to hyperplanes of $S$, their intersection lattices
are isomorphic and their characteristic polynomials are equal. We
will compute the characteristic polynomial $\chi_{S}(x)$ using the
finite field method. For a sufficiently large prime number $p$,
$\chi_{S}(p)$ is equal to the number of vectors $(x_1,...,x_n), x_i
\in F_p$, that do not satisfy any of the equations
\begin{enumerate}
\item $\forall i, 1\leq i\leq n: x_i = 0 $.
\item $\forall i_1, i_2, 1\leq i_1 \leq i_2 \leq n,(i_1, i_2) \neq (1,3) : x_{i_2} - x_{i_1} = 0$.
\item $\forall i, 4\leq i \leq n: x_i - x_1 + x_2 - x_3 = 0$.
\item $x_1 - x_2 + x_3 = 0$.
\end{enumerate}

Assume that values $x_1 = k_1,x_2 = k_2,x_3 = k_3$  are already
chosen. Then, the values for the rest of $x_i$ should be chosen, so, 
that none of the following conditions are satisfied
\begin{enumerate}
\item $\forall i, 4\leq i\leq n: x_i = 0 $.
\item $\forall i, 4\leq i\leq n: x_i = k_1 $.
\item $\forall i, 4\leq i\leq n: x_i = k_2 $.
\item $\forall i, 4\leq i\leq n: x_i = k_3 $.
\item $\forall i, 4\leq i\leq n: x_i = k_3 + k_1 - k_2 $.
\item $\forall i_1, i_2, 4\leq i_1 \leq i_2 \leq n: x_{i_2} - x_{i_1} = 0$.
\end{enumerate}

The number of possible ways  to choose the rest of $x_i$ depends on
the number of distinct values among  $0, k_1, k_2, k_3, k_3+k_1
-k_2$.

The value for $k_1$ can be chosen  in $p-1$ ways. Let us consider
the following cases.
\begin{itemize}
\item Case 1: $k_3 = k_1$. The conditions for
choosing a value for  $k_2$ are $k_2 \neq 0, k_2 \neq k_1, k_2 \neq
2k_1$, and we can choose $k_2$ in $p-3$ ways. In this case, there
are 4 distinct values among $0, k_1, k_2, k_3, k_3 + k_1 - k_2$,
therefore, there are $(p-1)(p-3)(p-4)...(p-n)$ ways to choose values
for $(x_1,..., x_n)$ in this case.

\item Case 2: $k_2 = 2k_1, k_3 = 3k_1$. There are
4 distinct values among $0, k_1, k_2, k_3, k_3 + k_1 - k_2,$
therefore, there are $(p-1)(p-4)(p-5)...(p-n)$ ways to choose values
for $(x_1,..., x_n)$ in this case.

\item Case 3: $k_2 = 2k_1, k_3 \neq 3k_1, k_3 \neq k_1$.
The conditions for choosing a value for $k_3$ are $k_3 \neq 0, k_3
\neq k_1, k_3  \neq k_2, k_3 \neq 3k_1$, and we can choose $k_3$ in
$p-4$ ways. There are 5 distinct values among $0, k_1, k_2, k_3, k_3
+ k_1 - k_2$, therefore, there are
$(p-1)(p-4)(p-5)(p-6)...(p-(n+1))$ ways to choose values for
$(x_1,..., x_n)$ in this case.

\item Case 4: $k_2 = k_{1}/2, k_3 \neq k_1$. The conditions
for choosing a value for  $k_3$ are $k_3 \neq 0, k_3 \neq k_1, k_3
\neq k_2, k_3 \neq k_2 - k_1$,  and we can choose $k_3$ in $p-4$
ways. There are 5 distinct values among $0, k_1, k_2, k_3, k_3 + k_1
- k_2$, therefore, there are $(p-1)(p-4)(p-5)(p-6)...(p-(n+1))$ ways to 
choose values for $(x_1,..., x_n)$ in this case.

\item Case 5: $k_2 \neq 2k_1, k_2 \neq k_{1}/2, k_3 = 2k_2 - k_1$.
The conditions for choosing a  value for $k_2$ are $k_2 \neq 0, k_2
\neq k_1,  k_2 \neq 2k_1, k_2 \neq k_{1}/2$, and we can choose $k_2$
in $p-4$ ways. There are 4 distinct values among $0, k_1, k_2, k_3,
k_3 + k_1 - k_2$, therefore, there are
$(p-1)(p-4)(p-4)(p-5)...(p-n)$ ways to choose values for $(x_1,...,
x_n)$ in this case.

\item Case 6: $k_2 \neq 2k_1, k_2 \neq k_{1}/2, k_3
\neq 2k_2 - k_1, k_3 \neq k_1$. The conditions for choosing a value
for $k_2$ are $k_2 \neq 0, k_2 \neq k_1, k_2 \neq 2k_1, k_2 \neq
k_{1}/2$, and we can choose $k_2$ in $p-4$ ways. The conditions for
choosing a  value for $k_3$ are $k_3 \neq 0, k_3 \neq k_1, k_3 \neq
k_2, k_3 \neq k_2 - k_1, k_3 \neq 2k_2 - k_1$, and we can choose
$k_3$ in $p-5$ ways. There are 5 distinct values among $0, k_1, k_2,
k_3, k_3 + k_1 - k_2$, therefore, there are
$(p-1)(p-4)(p-5)(p-5)(p-6)...(p-(n+1))$ ways to choose values for
$(x_1,..., x_n)$ in this case.
\end{itemize}
The value of $\chi_{MA(G)}(p)$ is the sum of numbers  from these
cases. As a result, if $G$ consists of a $K_3$ subgraph and a
``tail'' subgraph that are connected through a shared vertex, the
formula for  $\chi_{MA(G)}(x)$  is
$$\chi_{MA(G)}(x) =
(x-1)(x-3)(x-4)...(x-(n-2))(x-(n-1))(x-(n-1))(x-n). 
$$
\end{proof}

\section{Improper weight function problem}

Let $G(V,E),$  $|E|=n$ be a graph without loops, multiple edges and
isolated vertices. Let $F_p$ be the group of residue classes modulo
a prime number $p$. A weight function $W:E \rightarrow F_{p}^{k}$
with an arbitrary $k \in \mathbb{N}$ is called improper, if there is
a sequence of edges $(e_{i_1},...,e_{i_d})$ that form a simple path
or a simple even cycle in $G$, such that
\begin{eqnarray}
 W(e_{i_1}) - W(e_{i_2}) + W(e_{i_3}) - ... + (-1)^{d-1}W(e_{i_d}) = {\bf 0} \in F_{p}^{k}. 
\label{one}
\end{eqnarray}

Otherwise the weight function is called proper. The choice of names
is motivated by the following consideration. If a weight function on
the edges of graph $G$ is given, then the weight of a matching can
be defined as the sum of the weights of its edges. Furthermore,
there is a one-to-one correspondence between the matchings of graph
$G$ and the vertices of the matching polytope of graph $G$
\cite{edmonds}. Thus, based on the weight function on the edges of
graph $G$, it is possible to construct a weight function on the
vertices of the matching polytope of graph $G$. An important
property of the matching polytope is that any two vertices of this
polytope are connected with the edge, if and only if the symmetrical
difference of the corresponding matchings consists of exactly one
simple path of arbitrary length or one simple cycle of even length
\cite[Theorem 25.3]{sch}.

 If the original weight function on the
edges of the graph was proper, then for the resulting weight
function on the vertices of the polytope, any two vertices connected
by an edge of the polytope have different weights. Conversely, if
the initial weight function on the edges of the graph was improper,
there is at least one pair of vertices in the matching polytope that
are connected by an edge and have the same weight. In this case, the
path or the cycle that satisfies (\ref{one}) is the symmetrical
difference of two matchings that correspond to these two vertices.

Let us state the improper weight function  problem. For an arbitrary
graph $G(V,E), |E| = n,$ without loops, multiple edges and isolated 
vertices, an arbitrary prime number $p$, an arbitrary natural number
$k$ and an arbitrary weight function $W:E\rightarrow F_{p}^{k}$ the
problem is to establish, if function $W$ is an improper one.

\begin{theorem}
The improper weight function problem is NP-complete.
\end{theorem}

\begin{proof}
We will first prove that the problem belongs to the complexity class
NP. By definition, a recognition problem $A$ with possible answers
``Yes'' or ``No'' belongs to the class NP if there is a
nondeterministic algorithm that consists of a guessing stage and a
checking stage, such that when $A$ has an answer  ``Yes'', the
guessing stage returns a structure $S$, whose size depends
polynomially from the size of the input to $A$, and then the
checking stage takes $S$  and the input to $A$ as its own input and
verifies in polynomial time, that the answer to $A$ was ``Yes''.

For an improper weight function problem,   an input consists of
graph $G$ that contains $n$ edges and $b, b \in \mathbb{N}, b(b-1)/2
\leq n$ vertices, and a weight function  $W:E \rightarrow F_{p}^{k}$
that can be described with $n  k$ elements of $F_p$. The size of an
input to an improper weight function problem is $O(n  k)$. Given an
instance of an improper weight function problem with answer ``Yes'',
a nondetermininstic algorithm can guess the path
$(e_{i_1},...,e_{i_d})$ that satisfies (\ref{one}), and then compute
$W(e_{i_1}) - W(e_{i_2}) + W(e_{i_3}) - ... + (-1)^{d-1}W(e_{i_d})$ 
to verify, that the answer to the improper weight function problem
is ``Yes''. And since $d \leq n$, both the size of structure
returned on the guessing stage and the time of computation on the
checking stage are polynomial from the input to the imporper weight
function problem.

To show that improper weight function problem is an NP-complete one,
we now transform an arbitrary 3-satisfiability problem into an
improper weight function problem. Let $X = \{x_1,..., x_n\}$ be a
set of Boolean variables and let $U = \{u_1,...,u_m\}$ be set of
clauses making up an arbitrary instance of 3-satisfiability problem.
The graph $G$ for the improper weight function problem will be
constructed in the following way. For each variable $x_i$, there
will be a subgraph $Z_i$ of the graph $G$ that consists of one cycle
of length 4, with the set of edges $\{a_i, b_i, c_i, d_i\}$ and the
set of vertices $\{A_i, B_i, C_i, D_i\}$, such that $a_i =
(A_i,B_i), b_i = (B_i, C_i),c_i = (C_i, D_i), d_i = (D_i, A_i) $.

\begin{figure}[h]
     \centering
     \includegraphics[width=0.5\textwidth]{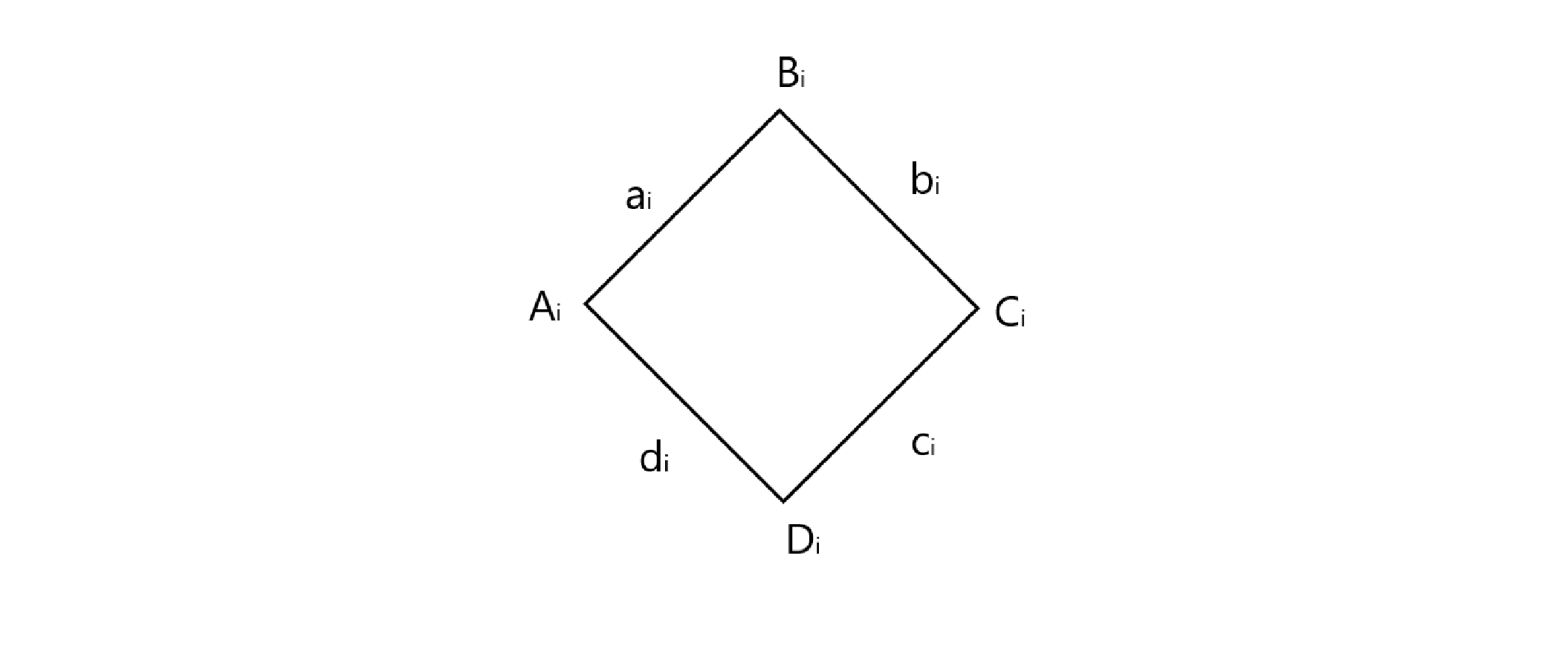}
     \caption{Subraph $Z_i$}
     \label{fig:my_label}
 \end{figure}

For each clause $u_j$, there  will be a bipartite subgraph $Y_j
\cong K_{2,7}$, with $\{P_j, Q_j\}$ and \\ $\{R_{j,1},
R_{j,2},R_{j,3}, R_{j,4}, R_{j_5}, R_{j,6}, R_{j_7}\}$ being two
parts of $Y_j$.

\begin{figure}[h]
     \centering
     \includegraphics[width=0.2\textwidth]{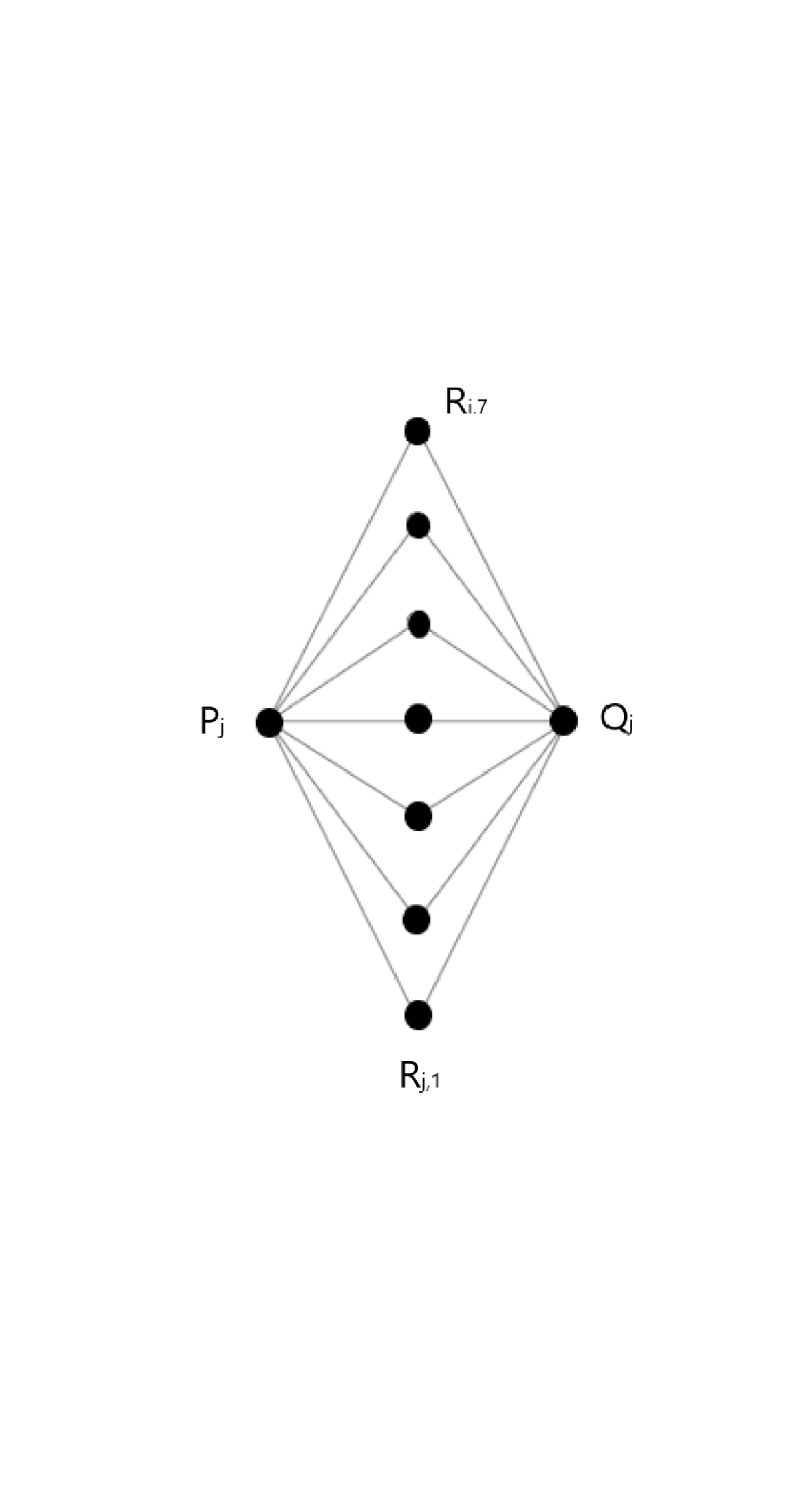}
     \caption{Subraph $Y_j$}
     \label{fig:my_label}
 \end{figure}

To construct the graph $G$, we connect $Z_i$  with $Z_{i+1}$ through
a shared vertex, starting with $Z_1$: $C_i = A_{i+1}$. Then, we
connect $Z_n$ with $Y_1$ through a shared vertex: $C_n = P_1$. Then,
we connect $Y_j$ with $Y_{j+1}$ through a shared vertex, starting
with $Y_1$: $Q_i =P_{i+1}$.

\begin{figure}[h]
     \centering
     \includegraphics[width=\textwidth]{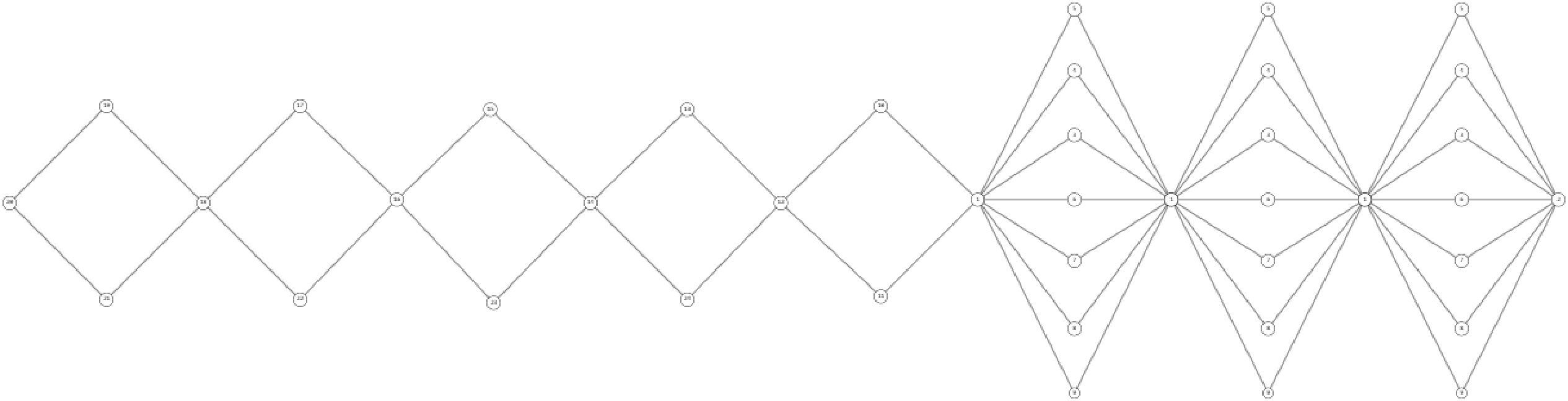}
     \caption{An example for $n=5$ and $m=3$}
     \label{fig:my_label}
 \end{figure}

The constructed graph $G(V,E)$ consists  of $3n + 8m +1$ vertices
and $4n + 14m$ edges. The next step is to construct the weight
function $W:E \rightarrow F_{p}^{k}$. The prime number $p$ can be
chosen arbitrarily. Let $k = 3n + 11m - 1$. The values of the weight
function $W(e)=(w_1(e),...,w_k(e))$ will be defined in the 
following way.

Let $\hat V = V\setminus \{A_1, Q_m\}$.  $|\hat V| = 3n +8m -1$. Let
all the elements of $\hat V$ be numbered: $\hat V =
\{v_1,...,v_{3n+8m-1}\}$. For $e \in E$ and $1 \leq i \leq 3n+8m-1$,
let $w_i(e) = 1$, if $e$ is an adjacent edge to $v_i$, and let
$w_i(e) = 0$ otherwise.
Let $u_j = (t_{\alpha},t_{\beta}, t_{\gamma})$ be a clause.

$w_{3n+8m-1+ 3(j-1)+1}(e)=1 $, if either $t_{\alpha} = x_{\alpha}$
and $e = (A_{\alpha},B_{\alpha})$, or $t_{\alpha} = \neg x_{\alpha}$
and $e = (A_{\alpha},D_{\alpha})$.

$w_{3n+8m-1+ 3(j-1) +1}(e) = -1$, if $e \in
\{(Q_j,R_{j,1}),(Q_j,R_{j,2}),(Q_j,R_{j,3}), (Q_j,R_{j,4})\}$.
Otherwise, $w_{3n+8m-1+ 3(j-1) +1}(e) = 0$.

$w_{3n+8m-1+ 3(j-1) +2}(e) =1$, if either $t_{\beta} = x_{\beta}$
and $e = (A_{\beta},B_{\beta})$, or $t_{\beta} = \neg x_{\beta}$ and
$e = (A_{\beta},D_{\beta})$.

 $w_{3n+8m-1+ 3(j-1) +2}(e) = -1$, if $e\in
 \{(Q_j,R_{j,1}),(Q_j,R_{j,2}),(Q_j,R_{j,5}), (Q_j,R_{j,6})\}$.
Otherwise, $w_{3n+8m-1+ 3(j-1) +2}(e) = 0$.  

$w_{3n+8m-1+ 3(j-1)+3}(e) =1$, if either $t_{\gamma} = x_{\gamma}$ and $e =
(A_{\gamma},B_{\gamma})$, or $t_{\gamma} = \neg x_{\gamma}$ and $e =
(A_{\gamma},D_{\gamma})$.

$w_{3n+8m-1+ 3(j-1) +3}(e) = -1$, if $e \in 
\{(Q_j,R_{j,1}),(Q_j,R_{j,3}),(Q_j,R_{j,5}), (Q_j,R_{j,7})\}$.
Otherwise, $w_{3n+8m-1+ 3(j-1) +3}(e) = 0$.

Now function $W$ is completely defined. Because of the way  that
$w_i(e), 1\leq i \leq 3n+8-1$, were defined, if a path starts or
ends in any vertex $v_j$ from $V \setminus \{A_1, Q_m\}$, there is
exactly one edge $e$ on this path that has $w_j(e) \neq 0$.
Therefore, it is impossible for a simple path
$(e_{i_1},...,e_{i_d})$ that satisfies (\ref{one}) to have any ends,
other then $A_1$ and $Q_m$.  Note that the only cycles of even
length in $G$ are cycles of length 4 inside $Z_i$ or $Y_j$. It is
not possible for a cycle of length 4 inside $Z_i$ or $Y_j$ to
satisfy (\ref{one}) as well, because with these values for $w_i(e),
i > 3n+8-1$, it is nessesary for every path or even cycle
$(e_{i_1},...,e_{i_d})$ that satisfies (\ref{one}) and contains an
edge in one of $Z_i$ to also contain an edge in one of $Y_j$, and
visa versa. In case of $Z_i$, the reason for this is that there is a
specific index $t$ that corresponds to a pair $(x_i, u)$, where $u$
is a clause that contains either $x_i$ or $\neg x_i$, such that
there is exactly one edge $e$ in $Z_i$ with $w_t(e) \neq 0$. In case
of $Y_j$, each of the edges $e \in
\{(Q_j,R_{j,1}),(Q_j,R_{j,2}),(Q_j,R_{j,3}), (Q_j,R_{j,4}),
(Q_j,R_{j,5}), (Q_j,R_{j,6}), (Q_j,R_{j,7})\}$ has its own distinct
vector of values for $(w_{3n+8m-1+ 3(j-1)+1}(e),w_{3n+8m-1+
3(j-1)+2}(e),w_{3n+8m-1+ 3(j-1)+3}(e))$, and since any cycle of
length 4 in $Y_j$ contains exactly two of those edges, it can't
satisfy (\ref{one}).
 Therefore, the only paths or even cycles that can possibly
satisfy (\ref{one}) are paths that connect $A_1$ and $Q_m$. Note
that $k$, the size of $G$ and the computation of $W$ are all
polynomial from $m$ and $n$.

Let $\{\alpha_1,...,\alpha_n\}$ be a set of  values assigned to $X$,
such that all clauses $u_1,...,u_n$ are satisfied. The current goal
is to find a path  $(e_{i_1}, ...,e_{i_d})$ in $G$ that satisfies
(\ref{one}), knowing $\{\alpha_1,...,\alpha_n\}$. Let us consider the 
following path $(e_{i_1},...,e_{i_{2(m+n)}})$ of length $2(m+n)$
from $A_1$ to $Q_m$. For $1\leq i \leq n$, if $\alpha_i = True$,
then $e_{2 i -1} = a_i, e_{2 i } = b_i$, and if $\alpha_i = False$,
then $e_{2 i -1} = d_i, e_{2 i } = c_i$.

Let $u_j = (t_{\alpha},t_{\beta}, t_{\gamma})$ be  a clause. Let us
consider all possible cases, after values $\alpha_1,...,\alpha_n$
are assigned.
\begin{itemize}
\item If $t_{\alpha} = True, t_{\beta} = True, t_{\gamma} = True$, then $e_{2n+2j-1} = (P_j, R_{j,1}), e_{2n+2j} = (Q_j, R_{j,1})$.
\item If $t_{\alpha} = True, t_{\beta} = True, t_{\gamma} = False$, then $e_{2n+2j-1} = (P_j, R_{j,2}), e_{2n+2j} = (Q_j, R_{j,2})$.
\item If $t_{\alpha} = True, t_{\beta} = False, t_{\gamma} = True$, then $e_{2n+2j-1} = (P_j, R_{j,3}), e_{2n+2j} = (Q_j, R_{j,3})$.
\item If $t_{\alpha} = True, t_{\beta} = False, t_{\gamma} = False$, then $e_{2n+2j-1} = (P_j, R_{j,4}), e_{2n+2j} = (Q_j, R_{j,4})$.
\item If $t_{\alpha} = False, t_{\beta} = True, t_{\gamma} = True$, then $e_{2n+2j-1} = (P_j, R_{j,5}), e_{2n+2j} = (Q_j, R_{j,5})$.
\item If $t_{\alpha} = False, t_{\beta} = True, t_{\gamma} = False$, then $e_{2n+2j-1} = (P_j, R_{j,6}), e_{2n+2j} = (Q_j, R_{j,6})$.
\item If $t_{\alpha} = False, t_{\beta} = False, t_{\gamma} = True$, then $e_{2n+2j-1} = (P_j, R_{j,7}), e_{2n+2j} = (Q_j, R_{j,7})$.
\end{itemize}
Now the path  $P = (e_{i_1},...,e_{i_{2(m+n)}})$ is completely 
defined.

Let
$$
W(e_{i_1}) - W(e_{i_2}) + W(e_{i_3}) - ... + (-1)^{2 
(m+n)-1}W(e_{i_{2 (m+n)}}) = (z_1,...,z_k) \in F_{p}^{k}.
$$
 For
$1 \leq i \leq 3n+8m-1$, $z_i = 0$, because for every vertex  $v \in
V\setminus \{A_1, Q_m\}$ path $P$ either passes through $v$ without
stopping there or doesn't pass through $v$ at all. For  $i >
3n+8m-1$, $z_i$ corresponds to a pair $(x_q, u_p)$, where $u_p$
contains either $x_q$ or~$\neg x_q$. There are four possible cases:
\begin{itemize} 
    \item Case 1: $\alpha_q = True$, $u_p$ contains $x_q$. In this case, $w_{i}(e_{i_{2 q -1}}) = 1, w_{i}(e_{i_{2 n+ 2 p -1}}) = -1, w_i(e) = 0$ for every other edge $e$ on $P$.
    \item Case 2: $\alpha_q = False$, $u_p$ contains $x_q$. In this case, $w_{i}(e_{i_{2 q -1}}) = 0, w_{i}(e_{i_{2 n+ 2 p -1}}) = 0, w_i(e) = 0$ for every other edge $e$ on $P$.
    \item Case 3: $\alpha_q = True$, $u_p$ contains $\neg x_q$. In this case, $w_{i}(e_{i_{2 q -1}}) = 0, w_{i}(e_{i_{2 n+ 2 p -1}}) = 0, w_i(e) = 0$ for every other edge $e$ on $P$.
    \item Case 4: $\alpha_q = False$, $u_p$ contains $\neg x_q$. In this case, $w_{i}(e_{i_{2 q -1}}) = 1, w_{i}(e_{i_{2 n+ 2 p -1}}) = -1, w_i(e) = 0$ for every other edge $e$ on $P$.
\end{itemize}
In each of these cases, $z_i = 0$. Therefore, for a path $P =
(e_{i_1},...,e_{i_{2 (m+n)}})$,  where $e_{i_1}$ is adjacent to
$A_1$ and $e_{i_{2 (m+n)}}$ is adjacent to $Q_m$, we have
$$
W(e_{i_1}) - W(e_{i_2}) + W(e_{i_3}) - ... + (-1)^{2 
(m+n)-1}W(e_{i_{2 (m+n)}}) = 0.
$$

Conversely, let $P = (e_{i_1},...,e_{i_{2 (m+n)}})$ be a path,
such that
$$
W(e_{i_1}) - W(e_{i_2}) + W(e_{i_3}) - ... + (-1)^{2 
(m+n)-1}W(e_{i_{2 (m+n)}}) = 0.
$$
 The current goal is to choose
values $\alpha_1,...,\alpha_n$ for $x_1,...,x_n$, such that all
clauses $u_1,...,u_m$ are satisfied. For $1 \leq i \leq n$, let
$\alpha_i = True$, if $e_{i_{2 i -1}} = a_i$, and let $\alpha_i =
False$, if $e_{i_{2 i -1}} = d_i$. The next step is to check if all
clauses from $U$ are satisfied. Let $u_j =
(t_{\alpha},t_{\beta},t_{\gamma})$ be a clause. If $e_{i_{ 2 n + 2 j
-1}} \in \{(P_j,R_{j,1}),(P_j,R_{j,2}),(P_j,R_{j,3}),(P_j,R_{j,4})
\}$, then because $W(e_{i_1}) - W(e_{i_2}) + W(e_{i_3}) - ... + 
(-1)^{2 (m+n)-1}W(e_{i_{2 (m+n)}}) = 0$, either $u_j$ contains
$x_{\alpha}$ and $e_{2\alpha -1} = a_{\alpha}$ , which means that
$x_{\alpha} = True$ and $u_j$ is satisfied, or $u_j$ contains $\neg
x_{\alpha}$ and $e_{2\alpha -1} = d_{\alpha}$, which means that
$x_{\alpha} = False$ and $u_j$ is satisfied again. In both cases,
$u_j$ is satisfied after values $\alpha_1,...,\alpha_n$ are
assigned. Cases when $e_{i_{ 2 n + 2 j -1}} \in
\{(P_j,R_{j,5}),(P_j,R_{j,6}),(P_j,R_{j,7}) \}$ can be resolved with
similar reasoning, with $x_{\beta}$ or $x_{\gamma}$ being considered
instead of $x_{\alpha}$. Therefore, after values $\alpha_1,...,
\alpha_n$ were chosen, all clauses $u_1,..., u_m$ are satisfied.

Therefore, this process of obtaining an improper  weight function
problem from a 3-SAT problem is a polynomial transformation of a
3-SAT problem into an improper weight function problem, which makes
an improper weight function problem an NP-complete one.
\end{proof}

\begin{proposition}
For an arbitrary graph $G(V,E)$ without loops  and multiple edges, a
sufficiently large prime number $p$ and an arbitrary natural $k$,
$\chi_{MA(G)}(p^k)$ is equal to the number of proper weight
functions $W:E \rightarrow F_{p}^{k}$.
\end{proposition}
\begin{proof}
Let $A$ be an arbitrary arrangement of hyperplanes in
$\mathbb{R}^n$, such that the equation of each hyperplane $A$ has 
integer coefficients. Let $A_q$ be a subset of all elements
$F_{q}^{n}$ satisfying the equation of at least one of the
hyperplanes $A$ modulo $p$, where $q = p^k$ is a prime power. The
finite field method allows us to find the value of the
characteristic polynomial of the arrangement $A$ at the point $q$
for sufficiently large $p$: $\chi_{A}(q) = |F_{q}^{n} \setminus
A_q|$.

Let $q = p^k$ be a power of   a prime number, and let $d$ be an
irreducible polynomial of degree $k$ over $F_p$. Then, there is a
one-to-one correspondence between $F_p^{k}$ and $F_q$ that preserves
addition and multiplication by elements of $F_p$: each element
$(a_1,..., a_k)$ from $F_{p}^{k}$ is associated with the element
$a_1 + a_{2}x + a_{3}x^{2}+...+a_{k}x^{k-1}$ in $F_p(x)/d \cong 
F_q$. This map induces a map from the set of functions $W:E
\rightarrow F_{p}^{k}$ to the set of functions $\hat W:E \rightarrow
F_{q}$.
 A proper function $W$ will be mapped to a function $\hat W$, which has the
following property: $(\hat W(e_1),..., \hat W(e_n))$ belongs to $
F_q^{n} \setminus MA(G)_q$, due to (\ref{one}) and the definition of
$MA(G)$.

Thus, for a sufficiently large  $p$ and an arbitrary $k$,
$\chi_{MA(G)}(p^k)$ is equal to the number of proper weight
functions $W:E \rightarrow F_{p}^{k}$.
\end{proof}

\section{An application of the improper weight function problem in~cryptography}

This section contains a description of a cryptosystem  that uses an
alternating weighted path problem, which is an extention of the
improper weight function problem. Given a graph $G(V,E), |E| = n$, a
prime number $p$, a natural number $k$, a weight function
$W:E\rightarrow F_{p}^{k}$ and a vector $(a_1,...,a_k) \in
F_{p}^{k}$, the problem is to determine, if there is a sequence of
edges $(e_{i_1},...,e_{i_d})$ that form a simple path or a simple
even cycle in graph $G$, such that
$W(e_{i_1})-W(e_{i_2})+...+(-1)^{d+1}W(e_{i_d})= (a_1,...,a_k)$.
Since we can use $(e_{i_1},...,e_{i_d})$ as a clue that would allow
to solve this problem in polynomial time, the problem is NP.
Moreover, we can reduce a general improper weight function problem
to a specific alternating weighted path problem by picking $G, p, k,
W$ as they were in the initial improper weight function problem and
choosing $(a_1,...,a_k) = {\bf 0} \in F_{p}^{k}$, making alternating
weighted path problem an NP-hard problem. Since alternating weighted
path problem is both NP and NP-hard, it's NP-complete.

Despite the general alternating weighted path problem  being
NP-complete, there are special cases of this problem that can be
solved fairly quickly. For example, let $G$ be a complete graph
$K_m$, let $p$ be a large prime, $p>2 \cdot 3^{m-1}$, and let $k =
m$. Let $V = \{v_1,...,v_m\}$ be a set of vertices of graph $G$. For
an edge $e = (v_q,v_r), q<r$, let the weight function $W(e) =
(w_1(e),...,w_m(e))$ be defined the following way: $w_t(e) =0$, if
$t \neq q, t\neq r$, $w_t(e) = 3^{r-q-1}$, if $t=q$, and $w_t(e) =
1$, if $t=r$. The key idea of the solution will be for a given $x
\in F_p$ to either find  an $1\leq i\leq m$ such that $x \equiv
3^{i-1} $(mod $p$) or  $x \equiv p-3^{i-1} $ (mod $p$), or find
$1\leq  i <j\leq m$, such that $x \equiv 3^{j} - 3^{i} $(mod $p$) or
$x \equiv p - 3^{j} + 3^{i} $ (mod $p$). Note that because $p > 2
\cdot 3^{m-1}$, and because for any $1\leq i <j \leq m$ and $r \in
\{0,1,3,...,3^{i-1}\}, s \in \{0,1,3,...,3^{j-1}\}$ $3^{j} - s >
3^{i} - r$, it is impossible to get multiple different answers for
an $x$ as a result of this search. Also note, that since
$|\{3^{i-1}|1\leq i \leq m\}\cup\{p-3^{i-1}|1\leq i \leq
m\}\cup\{3^{j}-3^{i}|1\leq i < j\leq m\}\cup\{p-3^{j}+3^{i}|1\leq i
< j\leq m\}| = m(m+1)$, this search can be done in polynomial time.

For a vector $A = (a_1,...,a_m) \in F_{p}^{k}$,  there is the
following solution to this problem. The solution consists of $m-1$
steps, and the $i$th step ends with either the conclusion, that the
answer is ``No'', or with the set of the edges $(v_i, v_j), i <j$
that belong to a path or an even cycle  $D = (e_{i_1},...,e_{i_d})$,
such that $W(e_{i_1})-W(e_{i_2})+...+(-1)^{d+1}W(e_{i_d})=
(a_1,...,a_m)$. In essense, this algoritm considers vertices of
graph $G$ one by one, from $v_1$ to $v_{m-1}$, tries to determine
for each vertex, if it has adjacent edges that are part of $D$, and
then adjusts the vector of weights, so that weights on edges of
vertices with lower indexes no longer affect the search of edges on
$D$, adjacent to vertices with higher indexes.

Let us describe the first step. If there is a path or an even  cycle
$D = (e_{i_1},...,e_{i_d})$, such that
$W(e_{i_1})-W(e_{i_2})+...+(-1)^{d+1}W(e_{i_d})= (a_1,...,a_m)$,
then there are 3 possible options. The first option is that $a_1 =
0$. In this case $D$ doesn't contain any edges that are adjacent to
$v_1$. The second option is that $a_1 = 3^b, b<m$ or $a_1 \equiv
p-3^{b}$ (mod $p$). In this case, $(v_1, v_{b+2})$ is in $D$ and it
is the only edge in $D$ that is adjacent to $v_1$. The third option
is that $a_1 \equiv 3^{b} - 3^{c}$ (mod $p$), $b<m, c<m$. In this
case, both edges $(v_1, v_{b+2})$ and $(v_1, v_{c+2})$ are in $D$
and they are the only edges in $D$ that are adjacent to $v_1$. If
$a_1$ doesn't fit into any of these three options, then the answer
to the alternating weighted path problem is  ``No''. Otherwise, the
set of edges that are adjacent to $v_1$ and are in the path $D$, is
stored, and the new vector of weights $B_1$ is obtained from $A$ in
the following way. If $a_1 = 0$ then the stored set of edges is
empty and $B_1 = A$. If $a_1 = 3^i$ then the edge $(v_1, v_i)$ is 
stored and $B_1 = A - W((v_1,v_i))$. If $a_1 \equiv p-3^i$ (mod
$p$), then the edge $(v_1, v_i)$ is stored and $B_1 = A +
W((v_1,v_i))$. If $a_1 \equiv 3^i - 3^j$ (mod $p$), then both edges
$(v_1, v_i)$ and $(v_1,v_j)$ are stored and $B_1 = A - W((v_1, v_i))
+ W((v_1,v_j))$. And this point the first step ends.

For the $i$th step the process is the same, except $v_i$ is used
instead of $v_1$, vector $B_{i-1}= (b_1,...,b_m)$, that we obtained
at the end of $(i-1)$th step, is used instead of A, $b_i$ is used
instead of $a_1$, and a new vector $B_i$ is obtained as a result
instead of $B_1$.

If all the steps are completed,  and none of them ended with the
conclusion that the answer is ``No'', then the set of all stored
edges is inspected. If all edges from this set form a simple path or
a simple even cycle $D$, and it fits the condition
$W(e_{i_1})-W(e_{i_2})+...+(-1)^{d+1}W(e_{i_d})= (a_1,...,a_m)$,
then the answer to the problem is  ``Yes'', otherwise the answer is
``No''.

Now let us describe a cryptosystem based on the alternating weighted path problem.
\begin{itemize}
    \item {\bfseries Public Key}: Graph $G(V,E)$, prime number $p$, natural number $k$, and weight function $W:E\rightarrow F_{p}^{k}$.
    \item {\bfseries Private Key}: An invertible linear operator $M: F_{p}^{k} \rightarrow F_{p}^{k}$ that transforms weights on edges, such that an alternating weighted path problem transforms into an easy one.
    \item {\bfseries Message Space}: simple paths and simple even cycles of $G$.
    \item {\bfseries Encryption}: for a path or an even cycle $(e_{i_1},...,e_{i_d})$, compute $(a_1,...,a_k) = W(e_{i_1})-W(e_{i_2})+...+(-1)^{d+1}W(e_{i_d})$.
    \item {\bfseries Decryption}: apply operator $M$ to the received message
    $A = (a_1,...,a_k)$ and solve the easy alternating weighted path problem
    with weight function $W'(e) = M(W(e))$ and the alternating sum of weights $B = M(A)$.
\end{itemize}

The example below shows a public key,  a private key, and a weight
function obtained after using private key.

\begin{figure}[h]
     \centering
     \includegraphics[width=0.65\textwidth]{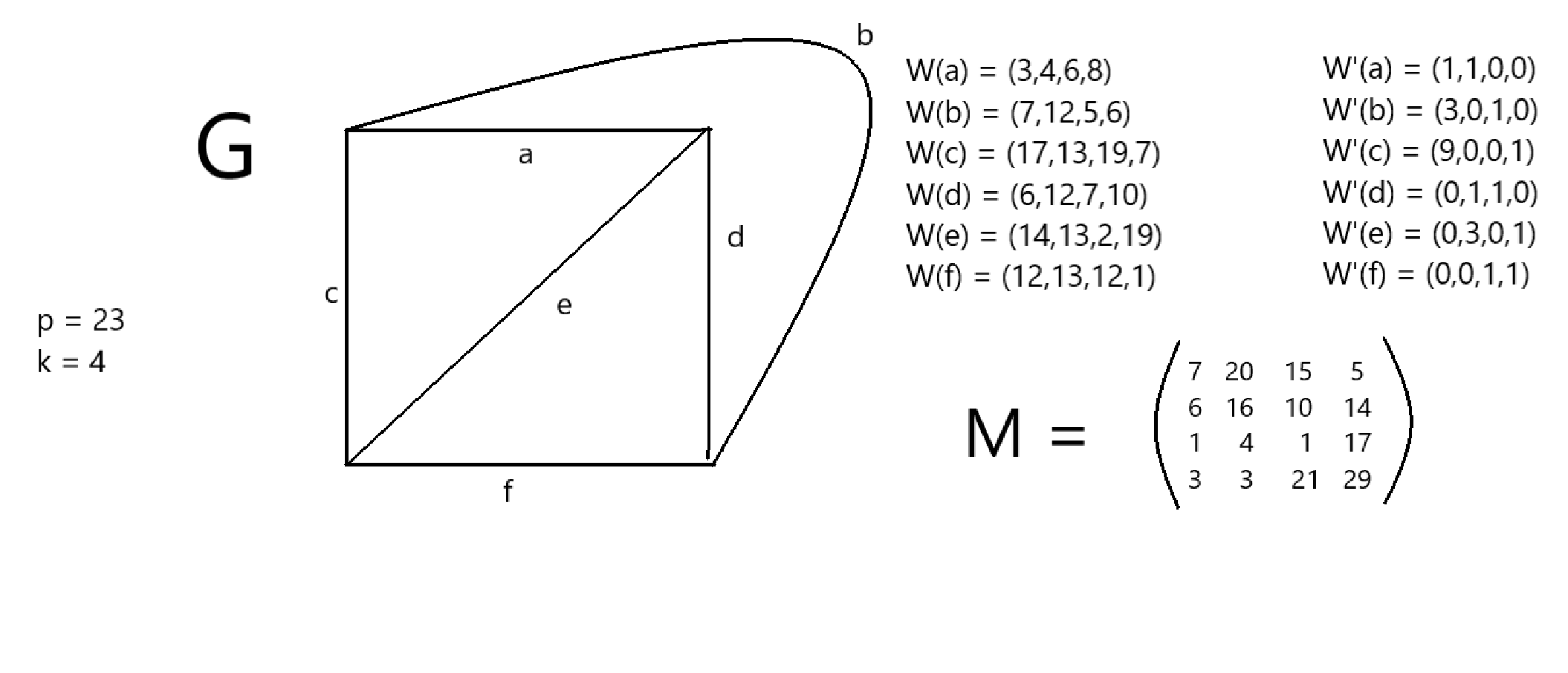}
     \caption{Example}
     \label{fig:my_label}
 \end{figure}

Such a system might require some  additions, so, it can be used in
transmissions of different kinds of messages. For example, a map
from the set of edges to the set of letters from English alphabet
can be added to the public key, so, that the system can be used to
transmit English words. To further explain this idea, let us add a
map $t: t(a) =$ ``A'', $t(b) =$ ``I'', $t(c) = $``N'', $t(d) =$
``K'', $t(e) = $``O'', $t(f) = $``M'' to the example above. If the
sender decides to send the message ``NO'', he chooses the path
$(c,e)$ that corresponds to this message according to the map $t$,
computes $W(c) - W(e) = (3, 0, 17, -12) $ (mod 23), and sends the
message $A = (3, 0, 17, -12)$. The receiver then computes $M(A) =
(9,20,0,0)$ (mod 23). On the first step of the solution of the easy
alternating path problem edge $c$ is stored and vector $(0,-3,0,-1)$
is obtained. On the second step of the solution edge $e$ is stored
and vector $(0,0,0,0)$ is obtained. After that, no more edges will
be added through the easy alternating path problem solution, and
after applying $t$ to the result $(c,e)$, the receiver obtains the
message ``NO''. 

One of the issues here is that the sender can use a
path $(e_{i_1},..,e_{i_d})$ to send either
$(t(e_{i_1}),...,t(e_{i_d}))$ or $(t(e_{i_d}),...,t(e_{i_1}))$. In
case when $(e_{i_1},..,e_{i_d})$, it is easy to draw a distinction
between these cases: if $A = W(e_{i_1}) - ... - W(e_{i_d})$, then
$(t(e_{i_1}),...,t(e_{i_d}))$ was sent, and if $A = W(e_{i_d}) - ...
- W(e_{i_1})$, then  $(t(e_{i_d}),...,t(e_{i_1}))$ was sent. But if
the path had odd length, the encryption process would give the same
result for $(t(e_{i_1}),...,t(e_{i_d}))$ and
$(t(e_{i_d}),...,t(e_{i_1}))$. Cycles of even length have a similar
problem: the encryption process would give the same result for
$(t(e_{i_1}),...,t(e_{i_d}))$ and $(t(e_{i_3}),...,t(e_{i_d}),
t(e_{i_1}), t(e_{i_2}))$.  This issue can be fixed by adding a
condition to the public key, that only messages based on paths of
even length can be sent. Since the path constructed in the proof of
NP-completeness had an even length, the addition of an even path
condition shouldn't make the cryptosystem more vulnerable to
attacks.

The important part of this cryptosystem is the superincreasing
sequence in the easy version of the alternating weighted path
problem. Similar approach is used in the Merkle--Hellman
cryptosystem that is based on the knapsack problem
\cite{merkle-hellman}. In the Merkle--Hellman cryptosystem, the
superincreasing sequence is disguised with a modular multiplication.
This way of disguise is known to be vulnerable to attacks
\cite{shamir} and \cite{br-odl}. In our cryptosystem, the
superincreasing sequence is disguised with a linear operator, which
is a broader method of disguise than the modular multiplication.
Thus there might be a variation of this cryptosystem that is less
vulnerable to attacks than knapsack-based cryptosystems.

\end{document}